\documentclass[12pt]{amsart}
\usepackage{amsmath,amscd,amssymb,amsfonts,latexsym}
\usepackage[all]{xy}
\usepackage{ulem}

\usepackage[latin1]{inputenc}

\usepackage[T1]{fontenc}

\usepackage{eucal} 
\usepackage{cmmib57} 
\usepackage{amsbsy}
\usepackage{amsthm,alltt}

\usepackage{graphicx,enumerate}

\usepackage{geometry}
\newtheorem{thm}{Theorem}[section]
\newtheorem{prop}[thm]{Proposition}

\newtheorem{cor}[thm]{Corollary}
\theoremstyle{definition}
\newtheorem{definition}[thm]{Definition}
\newtheorem{example}[thm]{Example}
\theoremstyle{remark}
\newtheorem{remark}[thm]{Remark}
\numberwithin{equation}{section}

\newcommand{\WLY}{\mathrm{WLY}}
\newcommand{\CN}{\mathbb{C}^{n+1}}
\newcommand{\U}{\mathcal{U}}
\newcommand{\N}{\mathbb{N}}
\newcommand{\Z}{\mathbb{Z}}
\newcommand{\C}{\mathbb{C}}
\newcommand{\w}{\textbf{w}}
\newcommand{\fst}{\textbf{1}}

\title{The total Milnor number of the weighted-L\^{e}-Yomdin-at-infinity polynomial}
\author{Yongqiang Liu}
\address{Y. Liu: School of Mathematical Sciences, University of Science and Technology of China, No.96,  JinZhai Road Baohe District, Hefei, Anhui, 230026, P.R.China}
\email {liuyq@mail.ustc.edu.cn}

\date{\today}



\begin{document}

\begin{abstract}    We give a concrete formula for the total Milnor number of the weighted-L\^{e}-Yomdin-at-infinity polynomial in most of the interesting cases. As an application, we give a description of the monodromy fibration at infinity for such kind of polynomial. 
 \end{abstract}

\maketitle

\section{Introduction}
Let $f\in \C[x_{1},\ldots,x_{n+1}]$ be a non-constant polynomial. Consider a system of positive weights $\textbf{w}=(w_{1},\ldots , w_{n+1}) \in \N^{n+1}$ with $\gcd(w_{1},\ldots , w_{n+1})=1$. Let us write $f=f_{N}+f_{N-1}+ \cdots + f_{0}$, where $f_{i}$ is the weighted degree $i$ part of $f$ and $f_{N}\neq 0$. It may happen that $f_{N-1}=0$, and so let $N-k$ be the greatest integer such that $f_{N-k}\neq 0$ with $0<k<N$. For $\w=\fst$, where $\fst=(1,1,\ldots,1)$, we say that the weights are usual.

We denote by $\partial f$ the gradient of $f$ \begin{center}
$\partial f: =(\partial f/\partial x_{1},\ldots,\partial f/\partial x_{n+1} )$.
\end{center} 

\begin{definition} Define the subset $S(f) $ in  $ \CN$ by $$S(f):= \lbrace x\in \CN \mid \partial f_{N}=0 , f_{N-k}=0\rbrace.$$ One says that $f$ is {\it weighted-L\^{e}-Yomdin-at-infinity} (or $\WLY_{\infty}$) for $(\w; N; k)$, if $S(f) \subset \lbrace 0\rbrace$, where $0$ is the origin in $\CN$.
\end{definition}

 Dimca showed (\cite{D3}, Lemma 9) that any $\WLY_{\infty}$ polynomial must be quasi-tame. If the weights are usual ($\w=\fst$), the total Milnor number of $f$ is computed by Artal-Bartolo, Luengo and Melle-Hern\'{a}ndez in \cite{ALM}, and the zeta function of monodromy operator at infinity for the generic fibre is computed by Gusein-Zade, Luengo and Melle-Hern\'{a}ndez in \cite{GLM}. If, in addition, $k = 1$, Garc\'{i}a L\'{o}pez and N\'{e}methi studied the monodromy  operator at infinity and the limit mixed Hodge structure of the generic fiber of $f$ by a series of papers: \cite{LN1, LN2, LN3}.

The aim of this paper is to give a concrete formula for the total Milnor number of $f$ with general weights in most of the interesting cases. This formula only involves the system of numbers $(\w;N;k)$ and some local invariants from the transversal singularities of $f_{N}$. As an application, we give a description of the monodromy fibration at infinity for the generic fibre of $f$. Moreover, we show that $f$ is indeed tame in most cases.

There exists a certain duality between the local setting of isolated hypersurface singularities and the global setting of weakly tame polynomials (cf. \cite{D5}). The corresponding {\it local weighted-L\^{e}-Yomdin singularities} is studied by Artal-Bartolo, Fern\'{a}ndez de Bobadilla, Luengo, and Melle-Hern\'{a}ndez  in \cite{AFLM}.  For the usual weights case $(\w=\fst)$, Luengo and Melle-Hern\'{a}ndez gave an explicit formula for the local Milnor number (see \cite[Theorem 2]{LM}).  This formula is very similar to the one for the global setting obtained in \cite[Theorem 2.1]{ALM}. For the general weighted case, it is interesting to compare  our formula (\ref{3.7}) provided later with the one for the local setting obtained in \cite[Theorem 3.2]{AFLM}.

In section 2, we introduce the tame and quasi-tame polynomials. In section 3, we use the second construction in \cite{DN} and Dimca's formula in \cite{D4} to give the formula for the total Milnor number of $f$. Section 4 is devoted to the applications.
 
 \textbf{Acknowledgments.} 
The author thanks Laurentiu Maxim and Alexandra Dimca for valuable comments during the preparation of this work.

\section{Tame and quasi-tame polynomial}
For any polynomial $f:\CN \to \C$, there exists a finite bifurcation set $ B_{f}\subset \C$ such that $f$ is a locally trivial fibration over $ \C\setminus {B_{f}} $. We set $F_{t}=f^{-1}(t)$. If $t\in{\C\setminus {B_{f}}} $, then $F=f^{-1}(t)$ is called the {\it generic fibre} of $f$. 

\begin{definition} \cite{B1,B2}   A polynomial $f:\CN \rightarrow \C$ is called tame, if there exists a constant real number $\varepsilon > 0$ and a compact set $K\subset \CN$ such that for any $x\in \CN \setminus K$, $\Vert \partial f \Vert > \varepsilon $.
\end{definition}
\begin{definition} \cite{N}   A polynomial $f:\CN \rightarrow \C$ is called quasi-tame, if there exists no sequence of points $ x^{j} \in \CN $ such that $\lim\limits_{j\rightarrow \infty}\Vert x^{j} \Vert = \infty$, $\lim\limits_{j\rightarrow \infty}\partial f = 0$ and the sequence $$f(x^{j})-\sum_{i=1}^{n+1}  x^{j}_{i}\dfrac{\partial f}{\partial x_{i}}(x^{j}) $$ has a finite limit for $j\to \infty$.
\end{definition}

\begin{remark} It is clear that any tame polynomial must be quasi-tame. For the example of quasi-tame polynomial, which is not tame, see\cite[Remark 7]{NZ} or Example \ref{ex4.3} provided later.
\end{remark}
 
The {\it total Milnor number} of $f$ is defined by the sum of the local Milnor numbers in all the singular points of $f$:  \begin{center}
$\mu(f)= \sum_{x\in \CN} \mu_{x}(f)$,
\end{center}  where $\mu_{x}(f)$ is the local Milnor number of $f$ at the point $x$. If $dim  \lbrace \partial f =0\rbrace \geq 1$, then $\mu (f)= \infty$ by convention.  N\'{e}methi showed that the quasi-tame polynomial exhibits the following good properties:
\begin{thm} \cite[Theorem 1.13]{N} If $f$ is  quasi-tame, then each fibre of $f$ has the homotopy type of a bouquet of $n$-spheres. In particular, for the generic fibre, we have $F\simeq \vee _{\mu(f)} S^{n}$, where $\mu(f)$ equals to the number of the spheres.
 \end{thm}

Dimca showed (\cite[Lemma 9]{D3}) that any $\WLY_{\infty}$ polynomial must be quasi-tame if the weights are usual ($\w=\fst$). Moreover, he and Paunescu mentioned that \cite[Lemma 9]{D3} can be easily extended to the general weighted case, see the last sentence in the second paragraph of \cite{DP} p. 1036.  Then, we have:
\begin{prop}
Any $\WLY_{\infty}$ polynomial must be quasi-tame.
\end{prop}

\section{Formulas for the total Milnor number}
In this section, we give a concrete formula for $\mu(f)$ in most of the interesting cases.

\subsection{}
First, we show that the total Milnor number $\mu(f)$ can be expressed in terms of the Euler characteristic numbers of two affine Milnor fibers, defined by two weighted homogeneous polynomials both with at most 1-dimensional singularities.

Assume that $f$ is $\WLY_{\infty}$, then $S(f) =\lbrace x\in \CN \mid \partial f_{N}=0 , f_{N-k}=0\rbrace\subset \lbrace 0\rbrace$.   The Affine Dimension Theorem (\cite[Proposition 1.7.1]{H}) gives that $\dim \lbrace \partial f_{N}=0 \rbrace \leq 1$. So $f_{N}$ is a weighted homogeneous polynomial with at most $1$-dimensional singularities.

Without loss of generality, we assume $0\in \C\setminus B_{f}$ for the entire section. Otherwise, instead of studying $f$, we can study the polynomial $f-c$, where $c\in \C\setminus B_{f}$. Set $F_{0}=f^{-1}(0)$ and $\U=\CN\setminus F_{0}$. $f$ is $\WLY_{\infty}$, hence quasi-tame. $F_{0}$ is a generic fibre of $f$ by the assumption, then $F_{0}\simeq \vee_{\mu(f)}S^{n}$. Therefore, \begin{center}
$\mu(f)=(-1)^{n+1}\chi(\U)$. 
\end{center} 

The following construction can be found in \cite[Section 5.4]{DN} (named the second construction in loc.cit.). The cyclic cover of $\U$ with degree $N$ is defined by \begin{center}
$\U^{N}=\lbrace (u,x) \in \C^{\ast}\times  \CN   \mid f(x)-u^{N}=0 \rbrace$.
\end{center} 
Introduce a new variable $x_{0}$  of weight $1$ and let $\widetilde{f}(x_{0},x)$ be the  homogenization of $f$ with respect to the system of weights $(1,\w)$.  Consider the Milnor fibre $\widetilde{F}$ given by $\widetilde{f}=1$. Then we have an embedding $j: \U^{N} \rightarrow \widetilde{F}$ given by $j(u,x)=(u^{-1},u\cdot x)$, where $\cdot$ denotes the multiplication  associated to the system of weights $\w$, given by: \begin{equation} \label{3.0}
u \cdot x=(u^{-w_{1}}x_{1}, \ldots, u^{-w_{n+1}}x_{n+1}).
\end{equation}
 The complement $\widetilde{F}\setminus j(\U^{N})$ is characterized by $\lbrace x_{0}=0 \rbrace$, hence it can be indentified with the affine Milnor fibre $F_{N}: =\lbrace f_{N}=1\rbrace$ (considered in $\CN$), which is defined by the top weighted degree component $f_{N}$. Then, we have \begin{equation} \label{3.2}
\mu(f)=(-1)^{n+1}\chi(\U)=\dfrac{(-1)^{n+1}}{N}\chi(\U^{N})=\dfrac{(-1)^{n+1}}{N}(\chi(\widetilde{F})-\chi(F_{N})).
\end{equation}

\medskip

Next, we show that $\widetilde{f}^{-1}(0)$ has at most $1$-dimensional singularities.  
 Note that we also have a similar isomorphism  $i: \C^{\ast}\times F_{0} \rightarrow \widetilde{f}^{-1}(0)  \setminus  f_{N}^{-1}(0)$  given by $i(u,x)=(u^{-1},u\cdot x )$, where $f_{N}^{-1}(0)$ is characterized by $\lbrace x_{0}=0 \rbrace$. Recall the assumption that $F_{0}$ is a generic fibre of $f$ (hence smooth), then $\widetilde{f}^{-1}(0)  \setminus  f_{N}^{-1}(0)$ is also smooth.  This shows that the singularities of $\widetilde{f}^{-1}(0)$ is contained in the hyperplane $\lbrace x_{0}=0\rbrace$ in $\C^{n+2}$.  So we only need to compute the singularities of $\widetilde{f}^{-1}(0)$ in this hyperplane. Note that  $\widetilde{f}=f_{N}+x_{0}^{k}f_{N-k}+\cdots+x_{0}^{N}f_{0}$. Then  \begin{center}
 $\dfrac{\partial \widetilde{f}}{\partial x_{0}}\mid_{x_{0}=0}=\left\{ \begin{array}{ll}
f_{N-1}=0, & k=1, \\
0, & 1<k<N,\\
\end{array}\right.$
 \end{center}
 \begin{center}
 $\dfrac{\partial \widetilde{f}}{\partial x_{i}}\mid_{x_{0}=0}=\dfrac{\partial f_{N}}{\partial x_{i}}=0,$  for $1
 \leq i \leq n+1$.
 \end{center}
 So, if $k>1$, then  $\widetilde{f}^{-1}(0)_{sing}$ can be identified with $f_{N}^{-1}(0)_{sing}$; and, if $k=1$, the definition of the $\WLY_{\infty}$ polynomial gives that $\lbrace x\in \CN \mid\partial f_{N}=0, f_{N-1}=0 \rbrace \subset \lbrace 0 \rbrace$, then $\widetilde{f}$ has at most isolated singularity. Altogether, we indeed showed that $\widetilde{f}^{-1}(0)$ has at most $1$-dimensional singularities.

\subsection{}
If $f_{N}^{-1}(0)$ has only isolated singularities at the origin, then so does $\widetilde{f}^{-1}(0)$ (see the computations above). Using the formula of the Milnor number in \cite[Theorem 1]{MO}, we have that \begin{equation} \label{3.3}
\chi(F_{N})=1+(-1)^{n}\prod_{i=1}^{n+1} \dfrac{N-w_{i}}{w_{i}},
\end{equation}   \begin{equation} \label{3.4}
\chi(\widetilde{F})=1+(-1)^{n+1}(N-1)\prod_{i=1}^{n+1} \dfrac{N-w_{i}}{w_{i}}.
\end{equation}  Then (\ref{3.2}) gives that $\mu(f)=\prod_{i=1}^{n+1} \dfrac{N-w_{i}}{w_{i}}$ in this case.

Next, we assume that $f_{N}^{-1}(0)$ has 1-dimensional singularities, namely \begin{center}
$f_{N}^{-1}(0)_{sing}=\lbrace x\in \CN \mid \partial f_{N}(x)=0 \rbrace= \lbrace 0 \rbrace \cup \bigcup_{j=1}^{p} \C^{\ast}\cdot a^{j}$
\end{center}
for some points $a^{j}\in \CN$, one in each irreducible component of $f_{N}^{-1}(0)_{sing}$. Here $\cdot$ denotes the multiplication  associated to the system of weights $\w$ given in (\ref{3.0}). 
To simplify the notation, we analyze one of these points, denoted by $a$. Let $G$ be the {\it isotropy group} of the point $a$ with respect to the $\C^{\ast}$-action on $\CN$ associated to the system of weights $\w$. This group is  the multiplicative group of the $d$-roots of unity, where \begin{center}
$ d= \gcd\lbrace w_{i} ; a_{i} \neq 0 \rbrace$,
\end{center}  for $a=(a_{1},\ldots,a_{n+1})$. The assumption $\gcd(w_{1},\ldots , w_{n+1})=1$ at the very beginning of this paper assures that a generic point in $\CN$ has a trivial isotropy group.

Let $H$ be a small $G$ invariant hyperplane transversal to the orbit $\C^{\ast}\cdot a$ at the point $a$. For example, we may assume that $a_{n+1}\neq 0$ and then take $H: x_{n+1}=a_{n+1}$. We identify the germs $(\C^{n},0)$ with $(H,a)$ via the isomorphism given by $(y_{1},\ldots,y_{n})\mapsto (y_{1}+a_{1}, \ldots, y_{n}+a_{n}, a_{n+1})$. The isolated hypersurface singularity 
$(Y,0)=(H\cap f_{N}^{-1}(0),a)$ is called the {\it transversal singularity} of $f_{N}$ along the branch $\C^{\ast}\cdot a$. Denote by $g$ the local function germ  of $Y$ at $0$.   

Let $\gamma$ be the generator of the isotropy group $G$. For any point $y=(y_{1},\ldots,y_{n})\in H$, $\gamma \ast y =\lambda\cdot y= (\lambda^{-w_{1}}y_{1}, \ldots, \lambda^{-w_{n}}y_{n})$, where $\lambda=e^{2 \pi i /d}$. The associated action of $G$ on any polynomial $h\in \C[y_{1},\ldots,y_{n}]$ is given by  \begin{center}
$\gamma \ast h(y)= h(\lambda^{-1}\cdot y)=h(\lambda^{w_{1}}y_{1}, \ldots, \lambda^{w_{n}}y_{n}).$
\end{center} In general, $g$ is not $G$-invariant, but its zero set $Y=g^{-1}(0)$ is so.

Let $\C[[ y_{1}, \ldots, y_{n} ]]$ denote the formal power series ring. $g$ defines an isolated singularity, then let  $\mu_{0}(g)$ be the local Milnor number of $g$ at 0, which is the length of the Milnor algebra \begin{center}
$M(g)=\C[[ y_{1}, \ldots, y_{n} ]]/ (\partial g/\partial y_{1}, \ldots, \partial g/ \partial y_{n})$,
\end{center}
and, $\tau_{0}(g)$  be the local Tjurina number of $g$ at 0, which is the length of the Tjurina algebra \begin{center}
$T(g)=\C[[ y_{1}, \ldots, y_{n} ]]/ (g,\partial g/\partial y_{1}, \ldots, \partial g/ \partial y_{n})$,
\end{center}  respectively.
 Suppose that the system of monomials of $\lbrace e_{1}, \ldots, e_{\mu_{0}(g)} \rbrace$ is a basis for $M(g)$, then $\gamma$ defines a linear transformation $T$ on the vector space $M(g)$. Since $\gamma^{d}=1$, then $T$ is diagonalizable and the corresponding eigenvalues have order $d$. Denote by $M(g)^{l} $ the eigenspace of $T$ associated to the eigenvalue $e^{2 \pi i l/d }$.

 \medskip
For a weighted homogeneous polynomial with 1-dimensional singularities, Dimca gave (\cite{D4}) an explicit formula for the Euler characteristic number of its Milnor fibre (conjecturally in all, but surely in most of the interesting cases). This formula is composed by the system of weights, the weighted degree and some local invariants from the transversal singularities.

Consider the Poincar\'{e} series \begin{center}
$P(t)=\prod_{i=1}^{n+1} \dfrac{1-t^{N-w_{i}}}{1-t^{w_{i}}}=\sum_{s\geq 0} c_{s}(\w;N)t^{s}$
\end{center} associated to $(\w;N)$. Define the {\it virtual Euler characteristics of order $m$} of $F_{N}$ by the formulas:\begin{center}
$\chi_{m}(\w;N)=1+(-1)^{n} \sum_{s=1}^{mN-w} c_{s}(\w;N),$
\end{center} where $w=w_{1}+\cdots +w_{n+1}$.

\begin{thm} \label{t3.1} \cite[Proposition 3.19]{D4}  Assume that $f_{N}$ is a weighted homogeneous polynomial with 1-dimensional singularities. Let $g_{j}$ (resp. $G_{j}$) be the local function germ (resp. the isotropy group) associated to the transversal singularity orbits $\C^{\ast}\cdot a^{j}$.
\item[(1)] If there is a weighted homogeneous polynomial of type $(\w;N)$ having only isolated singularity at the origin, then \begin{equation} \label{3.5}
 \chi(F_{N})=1+(-1)^{n}\prod_{i=1}^{n+1} \dfrac{N-w_{i}}{w_{i}}+(-1)^{n+1}\sum_{j=1}^{p} \sum_{s=1}^{N} \dim M(g_{j})^{-N-w+s}.
\end{equation}
\item[(2)] Assume that any transversal singularity $g_{j}=0$ for $j$ satisfies that  $\mu_{0}(g_{j})-\tau_{0}(g_{j})\leq 1$, then \begin{equation} \label{3.6}
\chi(F_{N})=\chi_{m}(\w;N)+(-1)^{n+1} \sum_{j=1}^{p} \sum_{s=1}^{N} \dim M(g_{j}) ^{(m-1)N-w+s}
\end{equation}
 for all $m$ large enough. 
\end{thm}
\begin{remark} \label{r3.1} Note that if there is a weighted homogeneous polynomial of type $(\w;N)$ having only isolated singularity at the origin, then \begin{center}
 $\chi_{m}(\w;N)=1+(-1)^{n}\prod_{i=1}^{n+1} \dfrac{N-w_{i}}{w_{i}}$   for any $m\geq n+1$,
\end{center}  see \cite{D4}, p. 40. Then, on a formal level, the formula (\ref{3.5}) is a special case of the formula (\ref{3.6}), obtained by taking $m$ divisible by all $d_{j}=\vert G_{j}\vert$, $1\leq j \leq p$.  Dimca conjectured that the formula (\ref{3.6}) holds for any transversal singularities.  For more results in this direction, see \cite{D4}.
\end{remark}

\subsection{}

To get the formula of $\mu(f)$, we need to analyze the transversal singularities of $ \widetilde{f}$.

If $k=1$,  $\widetilde{f}^{-1}(0) $ has only isolated singularity at the origin, then $\chi(\widetilde{F})$ is computed in (\ref{3.4}).

If $k>1$, $\widetilde{f}^{-1}(0) $ has $1$-dimensional singularities, and let $\widetilde{g} $ denote the transversal singularities at the orbits $\C^{\ast} \cdot \widetilde{a}$. Here $\widetilde{a}=(0,a)$ and we use the same notation  $\cdot$ to denote the multiplication  associated to the system of weights $(1,\w)$.   The isotropy group at the point $\widetilde{a}$ is the same as the one at the point $a$. We still take the hyperplane $H: x_{n+1}=a_{n+1}$, but considered in $\C^{n+2}$.  Then we identify the germs $(\C^{n+1},0)$ with $(H,\widetilde{a})$ via the isomorphism given by $(y_{0},y_{1},\ldots,y_{n})\mapsto (y_{0},y_{1}+a_{1}, \ldots, y_{n}+a_{n}, a_{n+1})$. The corresponding isolated hypersurface singularity 
$(\widetilde{Y},0)=(H\cap \widetilde{f}^{-1}(0),\widetilde{a})$
is defined by the function germ $\widetilde{g}$.  

Note that $\widetilde{g}$ is analytically isomorphic to $g+x_{0}^{k}$ at the point $\widetilde{a}$. To be more precise, $\widetilde{f}=f_{N}+x_{0}^{k}f_{N-k}+\cdots+x_{0}^{N}f_{0}= f_{N}+x_{0}^{k}(q+f^{\prime})$, where $q=f_{N-k}(a)\neq 0$ is a constant number and $f^{\prime}(\widetilde{a})=0$. $f_{N-k}(a)\neq 0$ follows from the assumption that $f$ is $\WLY_{\infty}$. Then,  \begin{center}
 $\dfrac{\partial \widetilde{f}}{\partial x_{0}}=x_{0}^{k-1}(kq+kf^{\prime}+x_{0}\dfrac{\partial f^{\prime}}{\partial x_{0}});$
 \end{center}
 \begin{center}
 $\dfrac{\partial \widetilde{f}}{\partial x_{i}}=\dfrac{\partial f_{N}}{\partial x_{i}}+x_{0}^{k}\dfrac{\partial f^{\prime}}{\partial x_{i}},$  for $1
 \leq i \leq n$.
 \end{center}
Note that $kf^{\prime}+x_{0}\dfrac{\partial f^{\prime}}{\partial x_{0}}$ takes value 0 at the point $\widetilde{a}$. So $kq+kf^{\prime}+x_{0}\dfrac{\partial f^{\prime}}{\partial x_{0}}$ is a unit in the ring  $\C[[ x_{0},x_{1}-a_{1},\ldots,x_{n}-a_{n} ]]$. Then the ideal generated by  
$(\partial\widetilde{f}/ \partial x_{0}, \ldots, \partial\widetilde{f}/ \partial x_{n})$ is the same as the one generated by $(x_{0}^{k-1},\partial f_{N}/ \partial x_{1}, \ldots, \partial f_{N}/ \partial x_{n})$. Therefore, a basis of $M(g+x_{0}^{k})$ remains a basis for the Milnor algebra $M(\widetilde{g} )$. Suppose that the system of monomials of $\lbrace e_{1}, \ldots, e_{\mu_{0}(g)} \rbrace$ is a basis for $M(g)$, then $\lbrace  x_{0}^{t}e_{s} \rbrace_{1\leq s \leq \mu_{0}(g)}^{ 0\leq t \leq k-2 }  $ is a basis for $M(g+x_{0}^{k})$, hence so is it for $M(\widetilde{g} )$. Altogether, we have that \begin{center}
$M(\widetilde{g} )^{l}= \oplus_{0\leq t \leq k-2} M(g)^{l-t}$.
\end{center}

The following theorem follows at once from the formula (\ref{3.2}), Theorem \ref{t3.1} and the above analysis for the transversal singularities of $\widetilde{f}$.
\begin{thm}  \label{t3.3} Let $f$ be a $\WLY_{\infty}$ polynomial with respect to $(\w; N; k)$ in $\CN$. 
\item[(1)] Assume that there is a weighted homogeneous polynomial for $(\w;N)$ with only isolated singularity, then so is it for $((1,\w);N)$, hence
\begin{equation} \label{3.7}
\mu(f)=\prod _{i=1}^{n+1} \dfrac{N-w_{i}}{w_{i}}-\dfrac{1}{N}\sum_{j=1}^{p} \sum_{s=1}^{N} \sum_{t=0}^{k-1} \dim M(g_{j})^{-N-w+s-t}.
\end{equation}
\item[(2)] Assume that any transversal singularity $g_{j}=0$ for $j$ satisfies that  $\mu_{0}(g_{j})=\tau_{0}(g_{j})$, then so does the singularity $g_{j}+x_{0}^{k}=0$, hence
\begin{equation} \label{3.8}
\mu(f)= (-1)^{n+1}\dfrac{\chi_{m}((1,\w);N)-\chi_{m}(\w;N)}{N} -\dfrac{1}{N}\sum_{j=1}^{p} \sum_{s=1}^{N} \sum_{t=0}^{k-1} \dim M(g_{j})^{(m-1)N-w+s-t}
\end{equation} for all $m$ large enough.
\end{thm}
\begin{remark} If $\mu_{0}(g_{j})-\tau_{0}(g_{j})=1$, then $\mu_{0}(g_{j}+x_{0}^{k})-\tau_{0}(g_{j}+x_{0}^{k})=k$. So, if $k>1$, the assumption in Theorem \ref{t3.1} (2) is not satisfied and we can not use the formula (\ref{3.6}) for $\widetilde{F}$. But we can  still compute $\mu(f)$ under some good assumptions in this case. For instance, it is possible that the weighted homogeneous polynomial with isolated singularity exists for $((1,\w);N)$, even if it is not the case for $(\w;N)$, see Example \ref{ex3.2} provided later.  Then the total Milnor number of $f$ can be computed by using formulas (\ref{3.5}) and (\ref{3.6}) for $\widetilde{F}$ and $F_{N}$, respectively.
\end{remark}

\begin{remark}
 For the usual weights case ($\w=\fst$), $G_{j}$ is trivial ($d_{j}=1$) for all $j$. Then $\dim M(g_{j})^{l}=\mu_{0}(g_{j})$ for any $l$. Note that the homogeneous polynomial with only isolated singularity always exists. Then (\ref{3.7}) gives $\mu(f)= (N-1)^{n+1}-k \sum_{j=1}^{p} \mu_{0}(g_{j})$. So we recover the formula obtained by Artal-Bartolo, Luengo and Melle-Hern\'{a}ndez in \cite[Theorem 2.1]{ALM}.
\end{remark}
\begin{cor} \label{c3.5} Let $f$ be a $\WLY_{\infty}$ polynomial with respect to $(\w; N; k)$ in $\CN$. If either $k=1$, or $f$ is in one of the two cases considered in Theorem \ref{t3.3}, then $\mu(f)$ is determined by the top weighted degree component $f_{N}$ and the number $k$. It does not depend on of the choice of $f_{i}$, for $i\leq N-k$. 
\end{cor}
\begin{proof}
For $k=1$, $\chi(\widetilde{F})$ is computed in (\ref{3.4}), involving only the system of numbers $(\w;N)$. It is clear that $\chi(F_{N})$ is determined by $f_{N}$. Then the claim follows from the formula (\ref{3.2}).

 Since $f_{N}$ contains all the information about the transversal singularities $g_{j}$ for all $j$, then the claim for the later case follows at once from Theorem \ref{t3.3}.
\end{proof}
\begin{remark} If Dimca's conjecture mentioned in Remark \ref{r3.1} holds for any transversal singularities, then  the formula (\ref{3.8}) and Corollary \ref{c3.5} also hold for any $\WLY_{\infty}$ polynomial.
\end{remark}

\begin{example} Assume that $f=x_{1}^{7}x_{2}+x_{3}^{3}+x_{2}$ in $\C^{3}$.  $f$ is $\WLY_{\infty}$ with respect to $\w=(2,1,5); N=15; k=14$.  Note that 
$f_{N}^{-1}(0)_{sing}:  x_{1}= x_{3}=0$.  So we can choose $a=(0,1,0)$ and the corresponding isotropy group $G$ is trivial $(d=1)$. The local function germ is $g=y_{1}^{7}+y_{3}^{3}$, and $\mu_{0}(g)=12$.  Note that there exists a weighted homogeneous polynomial of type $(\w; N)$ having only isolated singularity (for instance, $x_{1}^{7}x_{2}+x_{2}^{15}+x_{3}^{3}$).
So the formula (\ref{3.7}) gives  $\mu(f)=14$.

We can use another system of weights to check our results. $f$ is also $\WLY_{\infty}$ with respect to $\w^{\prime}=(1,2,3); N^{\prime}=9; k^{\prime}=7$.   $f_{N^{\prime}}^{-1}(0)_{sing}:  x_{1}= x_{3}=0$, $a=(0,1,0)$ and the local function germ is still $g=y_{1}^{7}+y_{3}^{3}$. But the corresponding isotropy  group $G$  becomes to $\Z/2\Z$, and $\dim M(g)^{l}=6$ for any $l$.  There also exists a weighted homogeneous polynomial of type $(\w^{\prime}; N^{\prime})$ having only isolated singularity (for instance, $x_{1}^{9}+x_{1}x_{2}^{4}+x_{3}^{3}$),   then the formula (\ref{3.7})  gives $\mu(f)=14$ again.
\end{example}

\begin{example} \label{ex3.2} Assume that $f$ is $\WLY_{\infty}$ in $\C^{4}$ with respect to $\w=(1,24,33,58); N=265$ and one fixed number $k<265$.  This homogeneity $(\w;N)$ is well-known, since the associated Poincare series is a polynomial in spite of the fact that the associated weighted homogeneous polynomial with isolated singularity do not exist, see \cite{AGV}, p. 201. 
Take the highest weighted degree form $f_{N}=x_{1}^{265}+x_{1}x_{2}^{11}+x_{1}x_{3}^{8}+x_{3}x_{4}^{4}$.  Note that $f_{N}^{-1}(0)_{sing}: x_{1}=x_{4}= x_{2}^{11}+x_{3}^{8}=0$. So we can choose $a=(0,-1,1,0)$, and the corresponding isotropy group is $G=\Z/3\Z$.  Dimca showed that the isotropy group acts on $M(g)$ such that $\dim M(g)^{l}=1$ for any $l$, and $\chi(F_{N})= -66250$, see \cite[Example 3.20]{D4}.  

For the homogeneity $(1,\w)=(1,1,24,33,58)$ and $N=265$, there exists an associated weighted homogeneous polynomial with isolated singularity (for instance, $2x_{0}^{265}+x_{0}x_{2}^{11}+x_{1}^{265}+x_{1}x_{2}^{11}+x_{1}x_{3}^{8}+x_{3}x_{4}^{4}$). Then the formula (\ref{3.5}) gives that  \begin{center}
$\chi(\widetilde{F})=1+264\times 264 \times \dfrac{265-24}{24}\times \dfrac{265-33}{33}\times \dfrac{265-58}{58}-265 (k-1)=17560490-265 k$.
\end{center}
 Using the formula (\ref{3.2}), we have $\mu(f)= 66516 -k > 66251$, where $0< k < 265$.
\end{example}

\section{Applications}

\subsection{Monodromy fibration at infinity} Let $B_{f}$ be the bifurcation set of $f$, then $f$ is a locally trivial fibration over $ \C\setminus {B_{f}} $. Fix $r$ such that $r> \max \lbrace \vert b\vert : b\in B_{f}  \rbrace$. The restriction \begin{center}
$ f: f^{-1}(S^{1}_{r})\to S^{1}_{r}$
\end{center} is called {\it the monodromy fibration at infinity} of $f$. The geometric monodromy associated with the path $s\rightarrow r e^{2\pi i s}, s\in [0,1]$ is a diffeomorphism of the generic fibre $F$ onto itself, which induces an isomorphism: \begin{center}
$M_{\infty}:\widetilde{ H}_{\ast}(F,\Z)\rightarrow \widetilde{H}_{\ast}(F,\Z)$,
\end{center}
This isomorphism is called the {\it  monodromy operator at infinity} of $f$. 

In this subsection, we show that the monodromy fibration at infinity of the $\WLY_{\infty}$ polynomial $f$ is only determined by $f_{N}$ and the number $k$ in most of the interesting cases.

 Recall the the following result obtained by H\`{a} Huy Vui and Zaharia:
\begin{thm} \label{t4.1} \cite{HZ} Let $f_{t}(x)$ be a family of polynomials, whose coefficients are smooth complex valued functions of $t\in [0,1]$. Suppose that for each $t\in [0,1]$, the polynomial $f_{t}(x)$ is M-tame and, moreover, the total Milnor number $\mu(f_{t})$ is independent of $t\in [0,1]$. Then the monodromy fibrations at infinity of the polynomials $f_{0}$ and $f_{1}$ are of the same fiber homotopy. If further $n\neq 2$, these fibrations are even differentiably isomorphic.
\end{thm}
\begin{remark}
M-tame polynomial is introduced by N\'{e}methi and Zaharia in \cite{NZ}. In particular, they showed that any quasi-tame polynomial must be M-tame in loc.cit..
\end{remark}

\begin{thm} \label{2.5} Fix a positive system of numbers $(\w;N;k)$. Let $f=f_{N}+f_{N-k}+ \cdots + f_{0}$ and $h=h_{N}+h_{N-k}+ \cdots + h_{0}$ be two $\WLY_{\infty}$ polynomials both with respect to $(\w;N;k)$, and $f_{N}=h_{N}$. If either $k=1$ or $f_{N}=h_{N}$ is in one of the two cases considered in Theorem \ref{t3.3}, then the monodromy fibrations at infinity of $f$ and $h$ are homemorphic equivalent. If further $n\neq 2$, these fibrations are even differentiably isomorphic.
\end{thm}
\begin{proof}
Assume that $f_{N}=h_{N}$ has the transversal singularities at the orbits $\C^{\ast} \cdot a^{j}$, for $1\leq j \leq p$.
 $f$ and $h$ being $\WLY_{\infty}$ imply that  $f_{N-k}(a^{j}) \neq 0$ and $h_{N-k}(a^{j}) \neq 0$ for any $j$.
 We can find a smooth path $\alpha: [0,1] \to \C$ such that   $\alpha(0)=1, \alpha(1)=0$ and \begin{equation} \label{4.1}
f_{N-k}(a^{j})\alpha(t)+ h_{N-k}(a^{j})(1-\alpha(t))\neq 0
\end{equation}   for any $t\in[0,1]$ and all $1\leq j \leq p$.
 Define $f_{t}(x)=\alpha(t) f + (1-\alpha(t)) h$ for $t\in[0,1]$.  The highest weighted degree form of $f_{t}(x)$ is fixed, which is just $f_{N}$. (\ref{4.1}) implies that $f_{t}(x)$ is $\WLY_{\infty}$ for any $t\in[0,1]$, hence quasi-tame.  Corollary \ref{c3.5} shows that the total Milnor number of $f_{t}(x)$ is independent of $t\in [0,1]$, then the claim follows from Theorem \ref{t4.1}. 
\end{proof}

\begin{example}  \label{ex4.1} Take $f_{1}= x_{1}^{3}+x_{1}x_{2}+x_{1}x_{3}^{2}+x_{3}x_{4}^{2}+x_{2}$ and $f_{2}= x_{1}^{3}+x_{1}x_{2}+x_{1}x_{3}^{2}+x_{3}x_{4}^{2}+x_{3}^{2}$  in $\C^{4}$. Then $f_{1}$ and $f_{2}$ are both $\WLY_{\infty}$ with respect to $\w=(1,2,1,1); N=3; k=1$, and they have the same highest weighted degree form. Note that there exists a weighted homogeneous polynomial of type $(\w; N)$ having only isolated singularity (for instance, $x_{1}^{3}+x_{1}x_{2}+x_{3}^{3}+x_{4}^{3}$). Then the monodromy fibrations at infinity of $f_{1}$ and $f_{2}$ are differentiably isomorphic.
\end{example}

\subsection{Tameness of the $\WLY_{\infty}$ polynomial}
In this subsection, we show that the $\WLY_{\infty}$ polynomial is indeed tame in most cases. 

Recall the following criterion of tame polynomial proved by Broughton: 
\begin{prop} \label{p4.1} \cite[Proposition 3.1]{B2}  Let $f: \CN\rightarrow \C$ be a  polynomial. For a point $v\in \CN$, define $f^{v}=f + \sum_{i=1}^{n+1} v_{i}x_{i}$ . Then f is tame if and only if $\mu(f)< \infty$ and $\mu(f)=\mu(f^{v})$ for all sufficiently small $v\in \CN$.
\end{prop}

\begin{cor} \label{c4.2} Let $f$ be a $\WLY_{\infty}$ polynomial with respect to $(\w; N; k)$ in $\CN$. If $f$ satisfies the assumption in Corollary \ref{c3.5} and $N-k\geq \max\lbrace w_{i} \rbrace$, then $f$ is tame. In particular, if the weights are usual $(\w=\fst)$, $f$ is tame for any $0<k<N$. 
\end{cor}
\begin{proof} Without loss of generality, assume that $w_{1}\geq w_{2} \geq \cdots \geq w_{n+1}$. Then the weighted degree of the linear polynomial $\sum_{i=1}^{n+1} v_{i}x_{i}$ is $w_{1}$.

If $N-k>w_{1}$, then $f^{v}$ is also $\WLY_{\infty}$ with respect to $(\w;N;k))$ and has the same highest weighted degree form as $f$. Corollary \ref{c3.5} gives that $\mu(f)=\mu(f^{v})$. So $f$ is tame by the previous proposition.

If $N-k=w_{1}$,  $\partial f_{N}=0$ are at most finite weighted lines passing through the origin in $\CN$. The affine hypersurface defined by $f_{N-k}=0$ does not contain any of these weighted lines defined by $\partial f_{N}=0$. After a small disturbance  induced by $\sum_{i=1}^{n+1} v_{i}x_{i}$, this is still true.  It follows that $f^{v}$ is still a $\WLY_{\infty}$ polynomial. Since $f^{v}$ has the same highest weighted degree form as $f$,   Corollary \ref{c3.5} gives that $\mu(f)=\mu(f^{v})$. By a similar argument, $f$ is tame.
\end{proof}

\begin{example} \label{ex4.3} Take the two polynomials considered in Example \ref{ex4.1}.   Note that there exists a weighted homogeneous polynomial of type ($\w=(1,2,1,1); N=3$) having only isolated singularity, and $N-k\geq w_{2}$. 
So $f_{1}$ and $f_{2}$ are tame. 

Now take $h= x_{1}^{3}+x_{1}x_{2}+x_{1}x_{3}^{2}+x_{3}x_{4}^{2}+x_{3}$ in $\C^{4}$. Then $h$ is $\WLY_{\infty}$ with respect to $\w=(1,2,1,1); N=3; k=2$. Choose a sequence of points $x^{n}= (\dfrac{1}{n^{2}}, -n^{4}/4, -n^{2}/2,0)$, then easy computation shows that  $\lim\limits_{n\rightarrow \infty}\Vert x^{n} \Vert = \infty$, $\lim\limits_{n\rightarrow \infty}\partial h = 0$.  So $h$ is not tame. Note that $N-k=1< max \lbrace w_{i} \rbrace=2$. This shows that the range  
 $N-k\geq \max\lbrace w_{i} \rbrace$ in Corollary \ref{c4.2} is optimal.
\end{example}
\subsection{Thom-Sebastiani Construction}
More complicated polynomial can be studied by using the Thom-Sebastiani Construction. Choose two polynomials $f\in \C[x_{1},\ldots,x_{n+1}]$ and $h\in \C[y_{1},\ldots,y_{m+1}]$ with two independent groups of variables. Assume that $f$ and $h$ are $\WLY_{\infty}$ with respect to $(\w;N;k)$ and $(\w^{\prime};N^{\prime};k^{\prime})$, respectively. In general, $f+h$ is not $\WLY_{\infty}$ anymore, see the following Example \ref{ex4.4}. 

Note that $\partial(f+h)(x,y)=0$  if and only if  $\partial f(x)=0$ and $\partial h(y)=0$. Thom-Sebastiani Theorem (cf. \cite[Corollary 3.3.21]{D1}) gives that $\mu_{(x,y)}(f+h)=\mu_{x}(f)\mu_{y}(h)$, then \begin{center}
$\mu(f+h)=\sum_{(x,y)} \mu_{(x,y)}(f+h)=(\sum _{x}\mu_{x}(f))( \sum_{y}\mu_{y}(h))=\mu(f)\mu(h)<\infty$.
\end{center} 

The generic fibers of $f$ and $h$ are $(n-1)$ and  $(m-1)$ connected, respectively.  Then the generic fibre of $f+h$ is $(n+m)$-connected, and it is homotopy equivalent to the join of the generic fibers of $f$ and $h$, see \cite{N2}.  So the generic fibre of $f+h$ is homotopy equivalent to a bouquet of $(n+m+1)$-spheres, i.e. $$(f+h)^{-1}(c)\simeq \vee_{\mu(f+h)}S^{n+m+1},$$ where $\mu(f+h)=\mu(f)\mu(h)$ equals to the number of the spheres and $c$ is generic for $f+h$. 

\medskip
N\'{e}methi obtained the  global Thom-Sebastiani theorem for the polynomial maps: the monodromy operators at infinity of the generic fibers satisfy the equality $M_{\infty}(f+h)=M_{\infty}(f)\otimes M_{\infty}(h)$, see \cite[Theorem 3.2]{N2}. 
 \begin{example} Take the two polynomials $f_{1}$ and $f_{2}$ considered in Example \ref{ex4.1}, then $M_{\infty}(f_{1})=M_{\infty}(f_{2})$.  So \cite[Theorem 3.2]{N2} shows that $M_{\infty}(f_{1}(x)+f_{1}(y))=M_{\infty}(f_{2}(x)+f_{2}(y))$.
 \end{example}
 For the property of tameness, we have the following result:
\begin{cor} Choose two polynomials $f\in \C[x_{1},\ldots,x_{n+1}]$ and $h\in \C[y_{1},\ldots,y_{m+1}]$ with two independent groups of variables. If $f$ and $h$ are both tame, then so is $f+h$. 
\end{cor}
\begin{proof}  If $f$ and $h$ are tame, then  $\mu(f+h)=\mu(f)\mu(h)<\infty$. For small enough $(v,v^{\prime})\in \C^{n+1}\times \C^{m+1}$, \[
\begin{aligned}
 \mu(f+h+\sum_{i=1}^{n+1} v_{i}x_{i}+ \sum_{j=1}^{m+1} v^{\prime}_{j}x_{j})
 & =\mu(f+\sum_{i=1}^{n+1} v_{i}x_{i})\mu(h+\sum_{j=1}^{m+1} v^{\prime}_{j}x_{j}) \\
 & =\mu(f)\mu(h)=\mu(f+h),
\end{aligned}
\]  where the third equality follows from Proposition \ref{p4.1}.  Using Proposition \ref{p4.1} again, the claim follows. 
\end{proof}
 
 \begin{example} \label{ex4.4} Take $f=x_{1}^{2}x_{2}+x_{3}^{3}+x_{2} $ and $h=y_{1}^{2}y_{2}+y_{3}^{3}+y_{2}$ both with usual weights and $N=3,k=2$. For any system of positive weights, the highest weighted degree form of $f+h$ always has at least 2-dimensional singularities, so $f+h$ is not $\WLY_{\infty}$.  Corollary \ref{c4.2} shows that $f$ and $h$ are both tame, then so is $f+h$.
 \end{example}


\end{document}